\documentclass[11pt]{article}
\usepackage{amsfonts}
\usepackage{cite}
\usepackage{amssymb}
\usepackage{amsthm}
\usepackage{amsmath}
\usepackage{graphicx}
\usepackage[OT2,T1]{fontenc}
\DeclareSymbolFont{cyrletters}{OT2}{wncyr}{m}{n}
\DeclareMathSymbol{\Sha}{\mathalpha}{cyrletters}{"58}
\title{Computing Congruence Primes Between a Newform With Integer Coefficients and the Old Space}
\date{}
\begin{document} 
\maketitle
\numberwithin{equation}{section}
\newtheorem{thm}{Theorem}
\newtheorem{df}{Definition}
\newtheorem{prop}{Proposition}
\newtheorem{conj}{Conjecture}
\newtheorem{lem}{Lemma}
\newtheorem{cor}{Corollary}
\newtheorem{alg}{Algorithm}
\newcommand{\Q}{\mathbb{Q}}
\newcommand{\R}{\mathbb{R}}
\newcommand{\Z}{\mathbb{Z}}
\newcommand{\T}{\textbf{T}}
\newcommand{\h}{\mathfrak{H}}
\newcommand{\C}{\mathbb{C}}
\newcommand{\F}{\mathbb{F}}

\begin{abstract}
In this article, we put together some known theoretical results and the fact that certain computations can be done efficiently in SAGE to come up with a fast algorithm for calculating congruence primes linking a newform with integer coefficients (i.e. a newform associated to an elliptic curve) with the old space at the same level. 
\end{abstract} 

\section{Introduction}

Let $N$ be a positive integer. Throughout this article, let $f$ be a newform of weight 2 and level $N$ with integer coefficients.  Let $X$ be some subspace of $S_{2}(\Gamma_{0}(N))$. Recall that we call a prime $p$ a \textit{congruence prime} between $f$ and $X$ if there exists some cusp form $g$ in $X$ with integer Fourier coefficients such that $a_{n}(f)\equiv a_{n}(g)$ mod $p$ for all $n$, (in which case, we write $f \equiv g$ mod $p$). Congruence primes play an important role in number theory. For instance, they were used extensively in the work of Ribet \cite{Rib90} and Wiles \cite{Wil95}, among others, leading to the proof of Fermat's Last Theorem.

Let $M_{X}$ be the set of forms in $X$ with integer coefficients. Suppose that one can efficiently compute a finite set $\{g_{1}\ldots g_{r}\}$ of forms in $M_{X}$ that generate $M_{X}$ as a $\Z$ module. (This is possible, for example, when $X$ is the old space.) In this article, we give a method for quickly computing congruence primes between $f$ and such a space $X$.

We became interested in computing congruence primes when trying to test two conjectures of Agashe regarding cancellations in the conjectural Birch and Swinnerton-Dyer formula for the elliptic curve associated to $f$ and their relation to congruences between $f$ and old forms, which we describe in detail in Section 5. Agashe was able to test these conjectures using the SAGE \cite{SAGE} command congruence\_number() for curves of conductor up to 1000, at which point computations became too slow to gather more data. The method we present is fast up to level 1500 and is fast up to much higher levels when the levels $N$ are nonsmooth (e.g. if $N$ is a product of two primes). In Sections 2 and 3, we present the algorithm for the general subspace $X$. In Section 4, we specialize to the case when $X$ is the old space, where it is especially easy to compute an integral spanning set for $M_{X}$. In Section 5, we explore the algorithm's motivation and main application and compare it to congruence\_number().

\section{Deciding if a Given Prime is a Congruence Prime}\label{deciding}


Recall that $f$ is a newform of level $N$ with integer coefficients, $X$ is a subspace of $S_{2}(\Gamma_{0}(N))$, and $M_{X}$ is the set of forms in $X$ with integer coefficients. In this section, we give a criterion for deciding if a given prime $p$ is a congruence prime between a newform $f$ and $X$. 

Let 
\begin{equation}\label{stmeq} B=\bigg{\lfloor} \frac{[SL_{2}(\Z):\Gamma_{0}(N)]}{6}-\frac{[SL_{2}(\Z):\Gamma_{0}(N)]-1}{N}\bigg{\rfloor}. \end{equation}

The following is a consequence of Theorem 1 in \cite{Stu87}.

\begin{thm}[Sturm]\label{sturmthm} Let $h$ and $g$ be two cusp forms in $S_{2}(\Gamma_{0}(N))$ and let $p$ be a prime. If $a_{n}(h)\equiv a_{n}(g)$ mod $p$ for all $n$ up to $B$ then $a_{n}(h)\equiv a_{n}(g)$ mod $p$ for all $n$.
\end{thm}

The integer $B$ above is called the \textit{Sturm bound} of $\Gamma_{0}(N)$. If $g$ is a cusp form, then let $v(g)$ denote the row vector whose components are the first $B$ coefficients of $g$. Let $\{g_{1},\ldots, g_{r}\}$ be a set of cusp forms with integer Fourier coefficients whose $\Z$-span is $M_{X}$. Let $M$ be the integer matrix whose $i$-th row is $v(g_{i})$ for $i$=1, \ldots, $r$.

\begin{lem}\label{biglemma}
The newform $f$ with integral coefficients is congruent modulo $p$ to a form in $X$ if and only if the $v(f)$ is in the row space of $M$ modulo $p$.
\end{lem}
\begin{proof}
If $v(f)$ is in the row space of $M$ modulo $p$, then we must have a row $w=\sum_{i=1}^{r} c_{i}v(g_{i})$ in the row space of $M$ such that $w_{n}\equiv v(f)_{n}$ mod $p$ for all $n$ up to $B$. Thus $a_{n}(\sum_{i=1}^{r}c_{i}g_{i})\equiv a_{n}(f)$ mod $p$ for all $n$ up to $B$. Theorem \ref{sturmthm} applies so that $a_{n}(\sum_{i=1}^{r}c_{i}g_{i})\equiv a_{n}(f)$ mod $p$ for all $n$.

Conversely, if there exists a form $g=\sum_{i=1}^{r}c_{i}g_{i}$ in $M_{X}$ such that $a_{n}(g)\equiv a_{n}(f)$, for all $n$, then $v(g)_{n}\equiv v(f)_{n}$ for all $n$ up to $B$. Then $v(f)\equiv \sum_{i=1}^{r} c_{i}v(g_{i})$ mod $p$, so that $v(f)$ is already in the mod-$p$ span of $M$.
\end{proof}

In section \ref{thisref}, we describe how one can efficiently compute $M$ associated to $X$ in the case that $X$ is the old space, so that Lemma \ref{biglemma} can be used to test if a prime is a congruence prime linking a newform to the old space.
\medskip


\section{Finding all Congruence Primes}
In this section, we describe an algorithm for finding all congruence primes between the newform $f$ and the subspace $X$ of $S_{2}(\Gamma_{0}(N))$ when one is given a finite set of forms that generate $M_{X}$ as a $\Z$ module. 

Let $E$ be the elliptic curve over $\Q$ associated to $f$. Consider the map $\phi:X_{0}(N) \rightarrow J_{0}(N)$ given by $\phi:P \mapsto (P)-(\infty)$. The map $\phi$ induces a surjective morphism $$X_{0}(N)\rightarrow J_{0}(N) \rightarrow E.$$ The degree of this composite morphism is called the \textit{modular degree} of $E$.
\medskip

The following is theorem 3.11 in \cite{ARS06}.
\begin{thm}[Agashe, Ribet, Stein]\label{bigthm}
Suppose $f$ is a newform of level $N$ with associated elliptic curve $E$, and $f$ is congruent modulo $p$ to a cusp form in its orthogonal complement in $S_{2}(\Gamma_{0}(N))$. Then either $p$ divides the modular degree of $E$ or $p^{2}|N$.
\end{thm}

The modular degree can be computed very efficiently. For a history of computational approaches of calculating the modular degree, see the introduction in \cite{Wat02}. In SAGE, Watkins's algorithm \cite{Wat02} for computing the modular degree of an elliptic curve is called by modular\_degree(). 

 
Lemma \ref{biglemma} provides an efficient test to check if a prime is a congruence prime between $f$ and the oldspace, while Theorem \ref{bigthm} gives a finite list of candidate primes that contains all congruence primes. Together, we have

\begin{alg}\label{alg} Given a newform $f$ of level $N$ and a spanning set for $X$ of forms with integer coefficients, this algorithm computes all congruence primes between $f$ and $X$.
\begin{enumerate}
\item Compute the matrix $M$ as described in Section \ref{deciding}.
\item Compute the modular degree and generate finite list of primes $\{p_{1}...p_{n}\}$ such that $p_{i}$ divides the modular degree or $p_{i}^{2}|N$. (Recall that by Theorem \ref{bigthm}, these are the only primes that could be congruence primes between $f$ and $X$.)
\item For each $p_{i}$, check if the $v(f)$ is in the row space of $M$ modulo $p_{i}$.
\item Output the list of primes for which the answer in step three is yes. 
\end{enumerate}
\end{alg}

\section{Computing the Matrix $M$ when $X$ Is the Old Space}\label{thisref}

Let $X$ be the old space at level $N$. We need to find a set of forms $\{g_{1},\ldots,g_{r}\}$ with integral coefficients that generate $M_{X}$ over $\Z$. 
Recall that for some $d|N$, the $d$-degeneracy map  is given by
\begin{equation}\label{dgns}
\beta_{d}:\sum_{i=1}^{\infty} a_{i}(f)q^{i} \mapsto \sum_{i=1}^{\infty} a_{i}(f)q^{di}.
\end{equation}

For each prime $p|N$, all old forms originating from any level dividing $\frac{N}{p}$ factor through $S_{2}(\Gamma_{0}(\frac{N}{p}))$ and pass to $S_{2}(\Gamma_{0}(N))$ by the $1$ and $p$ degeneracy maps.
Thus, to compute an integral spanning set for the old space at level $N$, it suffices to compute an integral basis for the spaces $S_{2}(\Gamma_{0}(\frac{N}{p}))$ for $p|N$ and pass these forms up to level $N$ by the 1 and $p$-degeneracy maps. 
This discussion leads to the following algorithm.
\begin{alg} Given a level $N$, this algorithm returns an integer matrix $M$ whose rows are $v(g_{i})$ where $\{g_{1},\ldots,g_{r}\}$ is a set of cuspforms with integer coefficients that spans the old subspace of $S_{2}(\Gamma_{0}(N))$.
\begin{enumerate}
\item Let $M$ be an empty matrix. 
\item For each prime $p|N$:
\begin{enumerate}
\item Compute an integral basis of $S_{2}(\Gamma_{0}(\frac{N}{p}))$:

\noindent In SAGE, the command integral\_basis() can be used to compute an integral basis for $S_{2}(\Gamma_{0}(\frac{N}{p}))$.
\item Compute the 1-degeneracy images from spaces $S_{2}(\Gamma_{0}(\frac{N}{p}))$:

\noindent Let $g$ be a form in a basis obtained in step (a). By equation \ref{dgns}, $\beta_{1}(g)$=$g$. So for each form $g$ in each the basis obtained in step (a), augment to $M$ the vector $v(g)$, which consists of the first $B$ coefficients of $g$. In SAGE, the command coefficients() can be used to compute the Fourier coefficients of a cusp form.
\item Compute the $p$-degeneracy images from spaces $S_{2}(\Gamma_{0}(\frac{N}{p}))$: 

\noindent By equation \ref{dgns}, the vector $v(\beta_{p}(g))$ is a vector of  length $B$ whose i-th component is zero if $p\nmid i$ and $v(g)_\frac{i}{p}$ if $p|i$. For each $g$ in the basis obtained in step (a), augment to $M$ the vectors $v(\beta_{p}(g))$ .
\end{enumerate}
\item Output the matrix $M$.
\end{enumerate}
\end{alg}

The reason that Algorithm 2 is fast in practice is because the command integral\_basis() in step 2(a) is very efficient. In Section 5, we indicate why Algorithms 1 and 2 compute congruences between $f$ and $X$ faster in SAGE than congruence\_number() when $X$ is the old space. 

\section{Motivation and Application}

In this section, we describe the problems that us to be interested in computing congruences between newforms associated to elliptic curves and the old space, and we indicate why our method is fast compared to congruence\_number() for this case. Let $E/\Q$ be a modular elliptic curve associated to the newform $f$ of level $N$. The right hand side of equation conjectural Birch and Swinnerton Dyer formula, 
$$ \frac{L^{r}(E/\Q,1)}{r! \,\,\Omega}=\frac{|\Sha(E/\Q)|\,\cdot\,\mathrm{R}(E/\Q)\,\cdot\,\prod_{p} c_{p}}{|E_{\mathrm{tors}}(\Q)|^{2}},$$ has been studied extensively (see the introduction in \cite{Lor10}). In particular, in regard to the relationship between the Tamagawa product $\Pi_{p}c_{p}$ and the order of $E_{\mathrm{tors}}(\Q)$, M. Emerton \cite{Eme03} showed that when $N$ is prime, $\Pi_{p}c_{p}=|E_{\mathrm{tors}}(\Q)|$. 

When $N$ is not prime, $|E_{\mathrm{tors}}(\Q)|$ need not equal $\Pi_{p}c_{p}$. In fact, this first occurs when $N=42$. D. Lorenzini \cite{Lor10} has shown that 
if $\ell>3$ is a prime number, then the order of the $\ell$-primary part of $E_\mathrm{tors}(\Q)$ divides $\Pi_{p}c_{p}$. 

In the other direction, Agashe \cite{Aga11} has conjectured that:

\begin{conj}[Agashe]\label{conj1}
If an odd prime $\ell$ divides $\Pi_{p} c_{p}$, then either $\ell$ divides the order of $E_{\mathrm{tors}}(\Q)$ or the newform $f$ is congruent modulo $\ell$ to a form in the old space.
\end{conj}

The following partial result towards Conjecture \ref{conj1} is given in \cite{Aga11}.

\begin{prop}
Let $\ell$ be an odd prime such that either $\ell \nmid N$ or for all primes r that divide $N$, $\ell \nmid (r-1)$. If $\ell$ divides the order of the geometric component group of $E$ at $p$ for some prime $p||N$, then either $E[\ell]$ is reducible or the newform $f$ is congruent to a newform of level dividig $N/p$ (for all Fourier coefficients whose indices are coprime to $N\ell$) modulo a prime ideal over $\ell$ in a number field containing the Fourier coefficients of both newforms. 
\end{prop}

In fact, in \cite{Aga11} the author conjectures more specifically that:

\begin{conj}[Agashe]\label{conj2}
If an odd prime $\ell$ divides $c_{p}$ for some prime $p$, then either $\ell$ divides the order of $E_{\mathrm{tors}}(\Q)$ or the newform $f$ is congruent modulo $\ell$ to a form in the subspace generated by degeneracy map images of newforms of levels dividing $N/p$.
\end{conj}

We became interested in computing congruence primes when trying to test Conjectures 1 and 2. Agashe was able to test these conjectures using the SAGE \cite{SAGE} command congruence\_number() for curves of conductor up to 1000, at which point computations became too slow to gather more data. The method we present is fast up to level 1500 and is fast up to much higher levels when the levels $N$ are nonsmooth (e.g. if $N$ is a product of two primes).

The algorithm used in congruence\_number() in SAGE is designed to calculate the congruence number between any two disjoint subspaces $X$ and $Y$ of $S_{2}(\Gamma_{0}(N))$. It does so by computing the integral structures, $X_{IS}$ and $Y_{IS}$, of the modular symbol spaces associated to $X$ and $Y$, and finding the index of $X_{IS} +Y_{IS}$ in its saturation - a process that requires: 
\begin{enumerate}
\item computing a $\Z$-basis for the modular symbol spaces corresponding to $X$ and $Y$
\item calculating the Hermite normal form $H$ of the matrix associated to $X_{IS}+Y_{IS}$
\item finding the Hermite normal form $M$ of $H^{T}$
\item taking the determinant of $M$.
\end{enumerate}

This method can be modified to check only if a prime $p$ is a congruence prime by performing steps 2-4 modulo $p$, but we do not know whether or not such an approach would detect all congruence primes without false negatives, and we do not know how efficient the modified algorithm would be. 

In the special case that $X$ is the span of a newform with integer coefficients and $Y$ is the old space, the method we described for computing congruence primes is faster than congruence\_number() for two reasons. First, the results in Section \ref{deciding} allow us to perform operations modulo certain primes. Second, as described in Section \ref{thisref}, when $Y$ is the old space, the problem of finding an integral basis for $Y$ can be reduced to computing integral bases of the cuspidal spaces at lower levels, which is already efficient in SAGE.

\bibliography{bib2}{}

\newcommand{\etalchar}[1]{$^{#1}$}
\providecommand{\bysame}{\leavevmode\hbox to3em{\hrulefill}\thinspace}
\providecommand{\MR}{\relax\ifhmode\unskip\space\fi MR }
\providecommand{\MRhref}[2]{%
  \href{http://www.ams.org/mathscinet-getitem?mr=#1}{#2}
}
\providecommand{\href}[2]{#2}
\begin{thebibliography}{ARS06}

\bibitem[Aga]{Aga11}
Amod Agashe, \emph{Unpublished manuscript}.

\bibitem[ARS06]{ARS06}
Amod Agashe, Kenneth Ribet, and William~A. Stein, \emph{The {M}anin constant},
  Pure Appl. Math. Q. \textbf{2} (2006), no.~2, part 2, 617--636. \MR{2251484
  (2007c:11076)}

\bibitem[Eme03]{Eme03}
Matthew Emerton, \emph{Optimal quotients of modular {J}acobians}, Math. Ann.
  \textbf{327} (2003), no.~3, 429--458. \MR{2021024 (2005g:11100)}

\bibitem[Lor]{Lor10}
Dino Lorenzini, \emph{Torsion and tamagawa numbers (preprint)}.

\bibitem[Rib90]{Rib90}
K.\thinspace{}A. Ribet, \emph{On modular representations of \protect{${\rm
  {G}al}(\overline{\bf{Q}}/{\bf {Q}})$} arising from modular forms}, Invent.
  Math. \textbf{100} (1990), no.~2, 431--476.

\bibitem[S{\etalchar{+}}11]{SAGE}
W.\thinspace{}A. Stein et~al., \emph{{S}age {M}athematics {S}oftware ({V}ersion
  4.6.2)}, The Sage Development Team, 2011, {\tt http://www.sagemath.org}.

\bibitem[Stu87]{Stu87}
J.~Sturm, \emph{On the congruence of modular forms}, Number theory (New York,
  1984--1985), Springer, Berlin, 1987, pp.~275--280.

\bibitem[Wat02]{Wat02}
Mark Watkins, \emph{Computing the modular degree of an elliptic curve},
  Experiment. Math. \textbf{11} (2002), no.~4, 487--502 (2003). \MR{1969641
  (2004c:11091)}

\bibitem[Wil95]{Wil95}
Andrew Wiles, \emph{Modular elliptic curves and {F}ermat's last theorem}, Ann.
  of Math. (2) \textbf{141} (1995), no.~3, 443--551.

\end{thebibliography}
\bibliographystyle{amsalpha}

\end{document}